\documentclass[11pt]{amsart}

\usepackage{graphicx}
\usepackage{enumitem}
\usepackage{lipsum}
\usepackage{verbatim}
\usepackage{url}
\usepackage{amssymb}
\usepackage{amsthm}
\usepackage{amsmath}
\usepackage[all, cmtip]{xy}
\usepackage{tikz}
\usepackage{dsfont}
\usepackage{spectralsequences}
\usepackage{mathtools}
\usetikzlibrary{matrix}
\usepackage[center]{caption}
\usepackage{stmaryrd}
\graphicspath{{pics/}}
\usepackage{tikz-cd}

\calclayout

\newlength\mylen
\newlist{mycases}{enumerate}{1}
\setlist[mycases,1]{label=\textbf{Case~\arabic*.}, 
  labelwidth=\dimexpr-\mylen-\labelsep\relax,leftmargin=0pt,align=right}

\newtheorem{question}{Question}
\newtheorem{theorem}{Theorem}

\newtheorem{lemma}{Lemma}[section]
\newtheorem{prop}[lemma]{Proposition}

\theoremstyle{definition}

\theoremstyle{remark}

\numberwithin{equation}{section}

\newcommand{\Z}{\mathbb Z}

\newcommand{\K}{\mathcal K}

\newcommand{\cI}{\mathcal I}

\newcommand{\p}[1]{\medskip \noindent \emph{#1}.}

\newcommand{\bE}{\mathbb E}

\newcommand{\calF}{\mathcal F}

\newcommand{\bn}{\noindent}

\newcommand{\calB}{\mathcal{B}}

\newcommand{\cut}{\ssearrow}

\let\oldunderset\underset
\protected\def\underset{\oldunderset}

\DeclareMathOperator{\Int}{Int}

\DeclareMathOperator{\Mod}{Mod}
\DeclareMathOperator{\Stab}{Stab}

\DeclareMathOperator{\Out}{Out}
\DeclareMathOperator{\Aut}{Aut}

\DeclareMathOperator{\im}{im}

\DeclareMathOperator{\cd}{cd}
\DeclareMathOperator{\vcd}{vcd}

\DeclareMathOperator{\UT}{UT}

\DeclareMathOperator{\Span}{Span}

\DeclareMathOperator{\degen}{degen}

\DeclareMathOperator{\proj}{proj}

\begin{document}

\title{THE COHOMOLOGICAL DIMENSIONS OF THE TERMS OF THE JOHNSON FILTRATION}

%    Information for first author
\author{Daniel Minahan}
%    Address of record for the research reported here
\address{686 Cherry Street NW, Atlanta, GA 30332}
%    Current address
\email{dminahan6@gatech.edu}
%    \thanks will become a 1st page footnote.

\date{\today}

\begin{abstract}
We prove that the $k$th term of the Johnson filtration of a closed, orientable surface of genus $g \geq 2$ has cohomological dimension $2g - 3$ for all $k \geq 3$ and $g \geq 2$.  This answers a question of Farb and Bestvina--Bux--Margalit.
\end{abstract}

\vspace*{-2.5cm}

\maketitle

\section{Introduction}\label{introsection}

Let $S_g$ be a closed orientable surface of genus $g$.  The mapping class group $\Mod(S_g)$ is the group of isotopy classes of orientation preserving diffeomorphisms of $S_g$.  Let $\pi_1^{(k)}(S_g)$ denote the $k$th term of the lower central series of $\pi_1(S_g)$, where $\pi_1^{(0)}(S_g) = \pi_1(S_g)$.  The $k$th \textit{term of the Johnson filtration of $\Mod(S_g)$} is the kernel of the map 
\begin{displaymath}
\Mod(S_g) \rightarrow\Out(\pi_1(S_g)/\pi_1^{(k)}(S_g))
\end{displaymath}
\bn and is denoted $\Mod_{(k)}(S_g)$.  The first term, $\Mod_{(1)}(S_g)$, is called the \textit{Torelli group} and is denoted $\cI_g$.  The second term, $\Mod_{(2)}(S_g)$, is called the \textit{Johnson kernel} and is denoted $\K_g$.  

The \textit{cohomological dimension} of a group $G$, denoted $\cd(G)$, is the supremum over all $n$ such that there is a $G$--module $M$ with $H^n(G;M) \neq 0$.  The simplest lower bound for $\cd(\Mod_{(k)}(S_g))$ is given by the maximal $n$ such that $\Z^n$ is a subgroup of $\Mod_{(k)}(S_g)$.  The largest known free abelian subgroup of $\Mod_{(k)}(S_g)$ for $k \geq 3$ is of rank $g - 1$ \cite{Farbproblems}.  We have the following theorem, which answers a question of Farb \cite[Problem 5.9]{Farbproblems} and Bestvina--Bux--Margalit \cite[Question 1.8]{BBM}.

\begin{theorem}\label{mainthm}
Let $g \geq 2$ and $k \geq 3$.  The cohomological dimension of $\Mod_{(k)}(S_g)$ is $2g - 3$.
\end{theorem}  

\bn If $G$ has torsion, the \textit{virtual cohomological dimension} of $G$ is $\cd(\Gamma)$ for any $\Gamma \subseteq G$ finite index and torsion free.  The quantity $\vcd(G)$ is indpendent of the choice of $\Gamma$ by a theorem of Serre \cite[pg 190]{Brownbook}.  The following is the known values for $\cd$ and $\vcd$ of the terms of the Johnson filtration:
\begin{itemize}
\item $\vcd(\Mod(S_g))  = 4g - 5$ by the work of Harer \cite{Harervcd},
\item $\cd(\cI_g) = 3g - 5$ by Bestvina--Bux--Margalit \cite[Theorem A]{BBM}, and
\item $\cd(\Mod_{(k)}(S_g)) = 2g - 3$ for $k \geq 2$.
\end{itemize}

\bn  The third result in the case that $k=2$ is also due to Bestvina--Bux--Margalit \cite[Theorem B]{BBM}.  

\p{Outline of the proof of Theorem \ref{mainthm}} Cohomological dimension is \textit{monotonic}, in the sense that $H \subseteq G$ implies $\cd(H) \leq \cd(G)$.  Our Theorem \ref{mainthm} is a corollary of monotonicity of cohomological dimension, Bestvina--Bux--Margalit's theorem that $\cd(\K_g) = 2g - 3$ as above \cite[Theorem B]{BBM}, and the following theorem.
\begin{theorem}\label{infgenthm}
Let $g \geq 2$ and $k \geq 3$.  The group $H_{2g-3}(\Mod_{(k)}(S_g);\Z)$ is infinitely generated.
\end{theorem}
\bn Theorem \ref{infgenthm} is a partial answer to the following question (see, e.g., Margalit \cite[Question 5.6]{Margalitproblems}).

\begin{question}\label{genquestion}
For which choices of natural numbers $g, k, j$ and commutative ring $R$ is the $R$--module $H_j(\Mod_{(k)}(S_g);R)$ finitely generated?
\end{question}

\p{Prior partial answers to Question \ref{genquestion}} Mess showed that $\cI_2$ is an infinitely generated free group \cite{Messfree}, which answers Question \ref{genquestion} for all $j, k$ and $R$ when $g = 2$.  Johnson showed that $\cI_g$ is finitely generated for $g \geq 3$, which answers Question \ref{genquestion} for $j = k = 1$ and any $R$ in the case that $g \geq 3$ \cite{JohnsonIII}.  Church--Ershov--Putman \cite[Theorem C]{CEP} showed that for $g \geq 2k - 1$, $\Mod_{(k)}(S_g)$ is finitely generated, which answers the question for $j = 1$ and any $R$ when $g \geq 2k - 1$.  In contrast, Bestvina--Bux--Margalit showed that $H_{3g-5}(\cI_g;\Z)$ is infinitely generated for all $g$ \cite[Theorem C]{BBM}.  Gaifullin showed that $H_{j}(\cI_g;\Z)$ is infinitely generated for $2g - 3 \leq j \leq 3g - 5$ \cite{Gaifullin}, and also showed that $H_{2g-3}(\K_g;\Z)$ is infinitely generated \cite{Gaifullinjohnson}.  

\p{The strategy of the proof of Theorem \ref{infgenthm}} Our proof of Theorem \ref{infgenthm} follows the same strategy of Bestvina--Bux--Margalit \cite{BBM}.  Gaifullin also used a variant of this approach to prove that $H_{2g-3}(\K_g;\Z)$ is infinitely generated \cite{Gaifullinjohnson}.  Bestvina--Bux--Margalit \cite{BBM} define a complex $\calB(S_g)$ called the complex of cycles.  They consider the action of $\K_g$ on $\calB(S_g)$ and the associated equivariant homology spectral sequence $\bE_{p,q}^r$ (see, e.g., \cite[Section VII.7]{Brownbook}).  The strategy is to apply the following inductive process for any $k \geq 3$:
\begin{enumerate}
\item Mess proved that $H_1(\Mod_{(k)}(S_2);\Z)$ is infinitely generated \cite{Messfree}.
\item We will prove that $H_{2(g-1) - 3}(\Mod_{(k)}(S_{g-1});\Z)$ infinitely generated implies that 
\begin{displaymath}H_{2g-3}(\Stab_c(\Mod_{(k)}(S_g));\Z)
\end{displaymath}
\bn is infinitely generated for $c \subseteq S_g$ a nonseparating simple closed curve.
\item Bestvina--Bux--Margalit \cite[Proposition 5.2]{BBM} show that, for any $k \geq 2$, there is an inclusion 
\begin{displaymath}
\bE_{0,2g-3}^1 \hookrightarrow H_{2g-3}(\Mod_{(k)}(S_g);\Z).
\end{displaymath}
\bn  By construction, the complex $\calB(S_g)$ contains a vertex represented by a nonseparating simple closed curve $c \subseteq S_g$.  By examining the spectral sequence $\bE_{p,q}^r$, we can show that there is a sequence of inclusions
\begin{displaymath}
H_{2g-3}(\Stab_c(\Mod_{(k)}(S_g));\Z) \hookrightarrow \bE_{0,2g-3}^1 \hookrightarrow H_{2g-3}(\Mod_{(k)}(S_g);\Z).
\end{displaymath}
\bn By Step (2), $H_{2g-3}(\Stab_c(\Mod_{(k)}(S_g));\Z)$ is infinitely generated.  Since the composition of injections is an injection, $H_{2g-3}(\Mod_{(k)}(S_g);\Z)$ is infinitely generated as well.  
\end{enumerate}

\p{The outline of the paper} In Section \ref{birmansection}, we will construct a variant of the Birman exact sequence for the Johnson filtration.  In Section \ref{homolsection}, we will prove in Proposition \ref{birmanhomolprop} that for certain vertices $c \in \calB(S_g)$, if Theorem~\ref{infgenthm} holds for $g' < g$ and $k \geq 3$, then $H_{2g-3}(\Stab_v(\Mod_{(k)}(S_g);\Z))$ is infinitely generated for every $k \geq 3$.  To do this, we will apply the Leray--Serre spectral sequence to the Birman exact sequence constructed Section \ref{birmansection}.  In Section \ref{mainsection}, we prove Theorem \ref{infgenthm} by the inductive process described above.  We then use Theorem \ref{infgenthm} to prove Theorem \ref{mainthm}.

\p{Acknowledgments} The author would like to thank his adviser Dan Margalit for suggesting this problem and for many helpful mathematical discussions.  The author would like to thank Lei Chen pointing out an error in an earlier version of Section \ref{birmansection}.

\section{Birman exact sequences for the Johnson filtration}\label{birmansection}

Our goal in this section is to prove Proposition \ref{oneboundsymplemma} below, which gives a variant of the Birman exact sequence for the Johnson filtration and verifies that this variant satisfies certain properties.

\p{Curves on a surface} For the remainder of this paper, a curve $c$ on a surface $S$ is an isotopy class of essential, simple, closed curves.  

\p{Partitioned surfaces} Let $S = S_{g}^b$ be an orientable compact surface of genus $g$ with $b$ boundary components. Putman \cite{Putmantorelli} defined the notion of a \textit{partitioned surface} $\Sigma = (S,P)$, which is a surface $S$ equipped with a partition $P$ of $\pi_0(\partial S)$.  Suppose that $S \subseteq S_h$, and this embedding has the property that for each connected component $S' \in \pi_0(S_h \setminus \Int(S))$, there is some $p \in P$ such that $\partial S' = p$.  Additionally, suppose that no connected component of $S_h \setminus \Int(S)$ is a disk or annulus.  The Torelli group $\cI(\Sigma)$, alternatively notated as $\cI(S,P)$, is the subgroup of $\Mod(S)$ generated by all elements $\varphi \in \cI_h$ such that $\varphi$ is supported on $S \subseteq S_h$.  Putman~\cite[Theorem 1.1]{Putmantorelli} showed that $\cI(S,P)$ is well--defined, i.e., independent of the choice of embedding $S \subseteq S_h$.  Later, Church \cite[Theorem 4.6]{Churchorbit} showed that the Johnson filtration of $\Sigma$, denoted $\Mod_{(k)}(\Sigma)$ or $\Mod_{(k)}(S,P)$, is also independent of the choice of embedding $S \subseteq S_h$.  Now, let $\Sigma = (S,P)$ be a partitioned surface, and let $\Sigma' = (S',P')$ be another partitioned surface given by capping off one boundary component $b \in S$ with a disk and forgetting the corresponding boundary component from the partition $P$.  Let $\UT S'$ denote the unit tangent bundle of $S'$, let $* \in S'$ be a basepoint, and let $\widetilde{*} \in \UT S'$ be a lift of $*$ to $\UT S'$.  Putman \cite[Theorem 1.2]{Putmantorelli} showed that the Birman exact sequence \cite{FarbMarg}
\begin{displaymath}
1 \rightarrow \pi_1(\UT S',\widetilde{*}) \rightarrow \Mod(S) \rightarrow \Mod(S') \rightarrow 1
\end{displaymath}
\bn restricts to a short exact sequence
\begin{displaymath}
1 \rightarrow \Gamma \rightarrow \cI(\Sigma) \rightarrow \cI(\Sigma') \rightarrow 1.
\end{displaymath}
\bn The subgroup $\Gamma \subseteq \pi_1(\UT S', *)$ is constructed as follows.  If $\varphi:H \rightarrow K$ is a homomorphism of groups, the \textit{graph of $\varphi$}, denoted $\Gamma(\varphi)$, is the subgroup of $H \oplus K$ generated by elements of the form $(h, \varphi(h))$.  There is a split short exact sequence
\begin{displaymath}
1 \rightarrow \Z \rightarrow \pi_1(\UT S', *') \rightarrow \pi(S', *) \rightarrow 1,
\end{displaymath}
\bn and the group $\Gamma$ in Putman's Birman exact sequence for the Torelli group is $\Gamma(\kappa)$ for some function $\kappa:F_1 \rightarrow \Z$, where $F_1 \subseteq \pi_1(S',*)$.  We will extend the Birman exact sequence for the Torelli group to $\Mod_{(k)}(\Sigma)$ in the following special case.

\begin{prop}\label{oneboundsymplemma}
Let $g \geq 3$.  Let $\Sigma_1 = (S_{g-1}^1, \pi_0(\partial S_{g-1}^1))$ and $\Sigma_2 = (S_{g-1}^2, \pi_0(\partial S_{g-1}^2))$.  Choose a boundary component $p \subseteq \partial S_{g-1}^2$, and let $\rho_1:\cI(\Sigma_2) \rightarrow \cI(\Sigma_1)$ be the map in Putman's Birman exact sequence for the Torelli group.  Choose a basepoint $* \in S_{g-1}^1$ and a lift of this basepoint to $\widetilde{*} \in \UT S_{g-1}^1$.  Let $\kappa: \pi_1(S_{g-1}^1,*) \rightarrow \Z$ be a map such that there is a subgroup $F_1 \subseteq \pi_1(S_{g-1}^1,*)$ with $\Gamma(\kappa|_{F_1}) \subseteq \pi_1(\UT S_{g-1}^1, \widetilde{*})$ equal to $\ker(\rho_1)$.  For any $k \geq 3$, there a subgroup $F_k \subseteq \left[\pi_1(S_{g-1}^1,*), \pi_1(S_{g-1}^1,*)\right]$ such that there is a short exact sequence
\begin{displaymath}
1 \rightarrow \Gamma(\kappa|_{F_k}) \rightarrow \Mod_{(k)}(\Sigma_2) \xrightarrow{\rho_k} \Mod_{(k)}(\Sigma_1)\rightarrow 1,
\end{displaymath}
\bn given by restricting the Birman exact sequence.  Furthermore, for any $k \geq 3$, the group $F_k$ has the following properties:

\begin{enumerate}
\item the quotient group $F_k/F_{k+1}$ is nontrivial, finitely generated and free abelian,
\item the induced action of $\Mod_{(k)}(S_{g-1,1})$ on $F_k/F_{k+1}$ is trivial, and
\item the quotient group $\pi_1(S_{g-1}^1,*)/F_k$ is solvable and non--cyclic.  
\end{enumerate}
\end{prop}

\bn There are several parts to the proof of Proposition \ref{oneboundsymplemma}.  We show in Lemma \ref{beswelldef} that the map $\rho_k$ given by restricting $\rho$ to $\Mod_{(k)}(\Sigma_2)$ has $\im(\rho_k) = \Mod_{(k)}(\Sigma_1)$.  We then use the higher Johnson homomorphisms of Morita to show that the group $F_k$ defined in the statement Proposition \ref{oneboundsymplemma} has properties (1)--(3) given in Proposition \ref{oneboundsymplemma}.  The relationship between $F_k$ and the higher Johnson homomorphisms is described in Lemma \ref{higherjohnsonlemmaprime}.  Properties (1) and (2) are verified in Lemma \ref{higherjohnsonlemma}, and property (3) is verified in the proof of Proposition \ref{oneboundsymplemma}.

\begin{lemma}\label{beswelldef}
Let $g \geq 3$ and $k \geq 2$.  Let $\Sigma_1 = (S_{g-1}^1, \pi_0(\partial S_{g-1}^1))$ and $\Sigma_2 = (S_{g-1}^2, \pi_0(\partial S_{g-1}^2))$.  Choose a boundary component $p \subseteq \partial S_{g-1}^2$.  Let $\rho:\Mod(S_{g-1}^2) \rightarrow \Mod(S_{g-1}^1)$ be the map in the Birman exact sequence given by capping off $p$ with a disk.  Let $\rho_k$ denote the restriction of $\Mod(S_{g-1}^2)$ to $\Mod_{(k)}(\Sigma_2)$.  Then $\im(\rho_k) = \Mod_{(k)}(\Sigma_1) \subseteq \Mod(S_{g-1}^1)$.  
\end{lemma}

\begin{proof}
Let $\iota:S_{g-1}^1 \rightarrow S_{g-1}^2$ be an embedding.  Let $\iota_*:\Mod(S_{g-1}^1) \rightarrow \Mod(S_{g-1}^2)$ denote the map given by extending $\varphi \in \Mod(S_{g-1}^1)$ by the identity map on $S_{g-1}^2 \cut \iota(S_{g-1}^1)$.  The map $\iota_*$ splits $\rho$ from the definition of the Birman exact sequence (see, e.g. \cite[Section 4.2]{FarbMarg}).  Furthermore, Church's work on the Johnson filtration \cite[Theorem 4.6]{Churchorbit} says that $\Mod_{(k)}(\Sigma_1) = \iota_*^{-1}(\Mod_{(k)}(\Sigma_2))$.  Since $\iota_*$ splits $\rho$, we therefore have $\im(\rho_k) = \Mod_{(k)}(\Sigma_2)$.
\end{proof}

\bn We now study the properties of the groups $F_k = \ker(\rho_k)$.

\p{Higher Johnson homomorphisms} Let $h \geq g$.  For each $j \geq 1$, there are maps $\tau_{j,h}:\Mod_{(j)}(S_h) \rightarrow A_{j,h}$ called the \textit{higher Johnson homomorphisms}, such that:
\begin{itemize}
\item the group $A_{j,h}$ is finitely generated and free abelian group, and
\item $\ker(\tau_{j,h}) = \Mod_{(j+1)}(S_h)$.  
\end{itemize}

\bn For $j = 1$ these maps were constructed by Johnson \cite{Johnsonhomomorphism}, and for $j \geq 2$ these were constructed by Morita \cite[Section 2]{Moritahomomorphism}.  We  use the higher Johnson homomorphisms to prove the following result.

\begin{lemma}\label{higherjohnsonlemmaprime}
Let $\Sigma = (S_g^b, P)$ be a partitioned surface with $g \geq 2$, and let $k \geq 1$.  Let $A_{k,\Sigma}$ be the cokernel of the inclusion map
\begin{displaymath}
\Mod_{(k+1)}(\Sigma) \rightarrow \Mod_{(k)}(\Sigma).
\end{displaymath}
\bn The group $A_{k,\Sigma}$ is non--trivial, finitely generated and free abelian.  
\end{lemma}

\begin{proof}
Let $\iota:\Sigma \rightarrow S_h$ be a map of partitioned surfaces \cite{Putmantorelli} such that no connected component of $S_h \cut \iota(S_g^b)$ is a disk or annulus.  We have a morphism of short exact sequences

\begin{center}
\begin{tikzcd}
1 \arrow[r] & \Mod_{(k+1)}(S_h) \arrow[r] & \Mod_{(k)}(S_h) \arrow[r] & A_{k,h} \arrow[r] & 1\\
1 \arrow[r] & \Mod_{(k+1)}(\Sigma) \arrow[r]\arrow[u, "\iota_*"] & \Mod_{(k)}(\Sigma) \arrow[r]\arrow[u, "\iota_*"] & A_{k,\Sigma} \arrow[r]\arrow[u] & 1\\
\end{tikzcd}
\end{center}

\bn By the work of Putman \cite[Theorem 1.1]{Putmantorelli} for the case $k = 1$ and Church for $k \geq 2$ \cite[Theorem 4.6]{Churchorbit} we have $\Mod_{(k)}(\Sigma) = \iota_*^{-1}(\Mod_{(k)}(S_h))$, and similarly for $k + 1$.  In particular, this implies that the left square in the diagram is a pullback square, so the right vertical arrow is injective by a diagram chase.  Since $A_{k,h}$ is finitely generated and free abelian by Johnson \cite{Johnsonhomomorphism} for $k = 1$ and Morita \cite{Moritahomomorphism} for $k \geq 2$, we see that $A_{k,\Sigma}$ is finitely generated and free abelian as well.
\end{proof}

We now prove that $F_k = \ker(\rho_k)$ satisfies properties (1) and (2) in Proposition \ref{oneboundsymplemma}.

\begin{lemma}\label{higherjohnsonlemma}
Let $g \geq 3$ and $k \geq 1$.  Let $\Sigma_1 = (S_{g-1}^1, \pi_0(\partial S_{g-1}^1))$ and $\Sigma_2 = (S_{g-1}^2, \pi_0(\partial S_{g-1}^2))$.  Choose a boundary component $p \subseteq \partial S_{g-1}^2$, and let $\rho:\Mod(S_{g-1}^2) \rightarrow \Mod(S_{g-1}^1)$ be the map in the Birman exact sequence.  Let $\rho_k$ denote the restriction of $\rho$ to $\Mod_{(k)}(\Sigma_2)$.  Let $\kappa:\pi_1(S_{g-1}^1,*) \rightarrow \Z$ be a map such that $\ker(\rho_1) = \Gamma(\kappa|_{K})$ for some $K \subseteq \pi_1(S_{g-1}^1,*)$, and let $F_k \subseteq \pi_1(S_{g-1}^1,*)$ be the group such that $\Gamma(\kappa|_{F_k}) = \ker(\rho_k)$.  For any $k \geq 3$, the group $F_k$ satisfies properties (1) and (2) of Proposition \ref{oneboundsymplemma}, namely
\begin{enumerate}
\item the group $F_k/F_{k+1}$ is finitely generated and free abelian and 
\item the action of $\Mod_{(k)}(\Sigma_1)$ on $F_k/F_{k+1}$ is trivial.
\end{enumerate}
\end{lemma}

\begin{proof}
Note that $F_k$ is canonically isomorphic to $\Gamma(\kappa|_{F_k})$ and the action of $\Mod_{(k)}(\Sigma_2)$ on the $\Z$--factor of $\pi_1(\UT S_{g-1}^1, *)$ is trivial, so we will work with $F_k$ instead of $\Gamma(\kappa|_{F_k})$.  Since $\rho_{k+1}$ is the restriction of $\rho_k$ to $\Mod_{(k+1)}(\Sigma_2)$, we have a short exact sequence of short exact sequences given by
\begin{center}
\begin{tikzcd}
& 1 & 1 & 1 & \\
1 \arrow[r] & F_k/F_{k+1} \arrow[r] \arrow[u] & A_{k, \Sigma_2} \arrow[r] \arrow[u] & A_{k, \Sigma_1} \arrow[r] \arrow[u] & 1\\
1 \arrow[r] & F_k \arrow[r]\arrow[u] & \Mod_{(k)}(\Sigma_2) \arrow[r, "\rho_k"]\arrow[u] & \Mod_{(k)}(\Sigma_1) \arrow[r]\arrow[u] & 1\\
1 \arrow[r] & F_{k+1} \arrow[r]\arrow[u] & \Mod_{(k+1)}(\Sigma_2) \arrow[r, "\rho_{k+1}"]\arrow[u] & \Mod_{(k+1)}(\Sigma_1) \arrow[r]\arrow[u] & 1.\\
& 1 \arrow[u] & 1\arrow[u] & 1\arrow[u] & \\
\end{tikzcd}
\end{center}

\bn We verify properties (1) and (2) in turn.

\p{Property (1)} Since $A_{k,\Sigma_2}$ and $A_{k,\Sigma_1}$ are finitely generated free abelian groups by Lemma \ref{higherjohnsonlemmaprime}, the group $F_k/F_{k+1}$ must be a finitely generated free abelian group as well.  

\p{Property (2)} To show that the action by $\Mod_{(k)}(\Sigma_1)$ on $F_k/F_{k+1}$ is trivial, recall from the argument in Lemma \ref{beswelldef} that $\rho_k$ is split.  Since $A_{k,\Sigma_2}$ is an abelian group, the conjugation action of $\Mod_{(k)}(\Sigma_2)$ on $A_{k,\Sigma_2}$ is the trivial action.  The fact that $\rho_k$ is split implies that $\Mod_{(k)}(\Sigma_1)$ acts trivially on $A_{k,\Sigma_2}$.  Since $F_{k}/F_{k+1}$ is a subgroup of $A_{k,\Sigma_2}$, the argument is complete.  
\end{proof}

\bn We are now ready to complete the section.

\medskip

\begin{proof}[Proof of Proposition \ref{oneboundsymplemma}]
By Lemma \ref{beswelldef} the claimed short exact sequence exists, and by Lemma \ref{higherjohnsonlemma} the kernel $F_k$ satisfies properties (1) and (2).  Putman has shown that for $k = 1$, the group $F_1 = [\pi_1(S_{g-1}^1,*), \pi_1(S_{g-1}^1,*)]$ \cite{Putmantorelli}. Therefore, it suffices to show that, for $k \geq 3$, $F_k$ satisfies property (3), namely that $\pi_1(S_{g-1}^1,*)/F_k$ is solvable and non--cyclic.  

\p{Property (3)} By Lemma \ref{higherjohnsonlemma}, the quotient group $\pi_1(S_{g-1}^1, *)/F_{k}$ is filtered by a sequence of subgroups such that the corresponding graded object consists of abelian groups, which in particular implies that $\pi_1(S_{g-1}^1,*)/F_k$ is solvable.  The fact that $\pi_1(S_{g-1}^1, *)/F_{k}$ is non--cyclic follows  the fact that $\pi_1(S_{g-1}^1,*)/F_1 \cong H_1(S_{g-1}^1;\Z)$.  Then $\pi_1(S_{g-1}^1, *)/F_{k} \twoheadrightarrow \pi_1(S_{g-1}^1,*)/F_1) \cong H_1(S_{g-1}^1;\Z)$, so $\pi_1(S_{g-1}^1, *)/F_{k}$ cannot be cyclic.
\end{proof}

\bn We will also record one more fact here, which will be useful in Section \ref{homolsection}.  

\begin{lemma}\label{conjactionlemma}
Let $F_k$ be as in Proposition \ref{oneboundsymplemma}.  Let $\proj: \Gamma(\kappa|_{F_k}) \rightarrow F_k$ denote the isomorphism induced by projection onto the $\pi_1(S_{g-1}^1, *)$--factor.  The pushforward map $\proj_*:H_1(\Gamma(\kappa|_{F_k});\Z) \rightarrow H_1(F_k)$ is an isomorphism of $\Mod_{(k)}(\Sigma_1)$--representations.
\end{lemma}

\begin{proof}
The action by conjugation of $\Mod_{(k)}(\Sigma_2)$ on $\pi_1(\UT S_{g-1}^1, \widetilde{*})$ is the restriction of the action of $\Mod(S_{g-1}^2)$ on $\pi_1(\UT S_{g-1}^1, \widetilde{*})$ (see, e.g., Farb and Margalit \cite[Fact 4.8]{FarbMarg}).  Then $\Mod_{(k)}(\Sigma_2)$ acts trivially on the kernel of the map $\pi_1(\UT S_{g-1}^1, \widetilde{*})\rightarrow \pi_1(S_{g-1}^1, *)$, so the lemma follows.  
\end{proof}

\section{The homology of curve stabilizers in the Johnson filtration}\label{homolsection}

The goal of this section is to prove Proposition \ref{birmanhomolprop} below.  This i part s the main part of the inductive step of the proof of Theorem \ref{infgenthm}. 

\begin{prop}\label{birmanhomolprop}
Let $g \geq 3$ and $k \geq 3$.  Suppose that $H_{2g-5}(\Mod_{(k)}(S_{g-1});\Z)$ is infinitely generated.  Let $c \subseteq S_g$ be a nonseparating simple closed curve.  Then $H_{2g-3}(\Stab_c(\Mod_{(k)}(S_g));\Z)$ is an infinitely generated abelian group.
\end{prop}
\bn The proof of Proposition \ref{birmanhomolprop} requires three ingredients:
\begin{enumerate}[ label=(\arabic{enumi})]
\item Lemma \ref{lerayserrelemma}, a general fact about the homology of a group extension,
\item Lemma \ref{nondegencoverlemma} and Lemma \ref{inforbit}, which are general facts about the homology of infinite covers of surfaces, and
\item Lemma \ref{akinlemma}, which says that $H_{2g-5}(\Mod_{(k)}(S_{g-1});\Z)$ infinitely generated implies \newline $H_{2g-4}(\Mod_{(k)}(S_{g-1,1});\Z)$ infinitely generated.
\end{enumerate}
\bn Lemma \ref{akinlemma} follows by applying Lemma \ref{lerayserrelemma}, Lemma \ref{nondegencoverlemma} and Lemma \ref{inforbit} to a variant of the Birman exact sequence given by Akin \cite{Akin}.  Proposition \ref{birmanhomolprop} will follow by applying the same argument structure to the short exact sequence constructed in Section \ref{birmansection}.

\p{Homology of group extensions} If $\langle \cdot, \cdot \rangle$ is a symplectic form on a free abelian group $V$, we say that $v \in V$ is degenerate if $\langle v, \cdot \rangle: V \rightarrow \Z$ is the zero map.  We say that a form $\langle \cdot, \cdot \rangle$ on $V$ is degenerate if some element of $V$ is degenerate, and otherwise that $\langle \cdot, \cdot \rangle$ is nondegenerate.  We will say that $V' \subseteq V$ is a \textit{degenerate subspace} if every $v \in V'$ is degenerate, and denote by $V^{\degen}$ the subspace spanned by all degenerate elements of $V$.  The following lemma is a generalization of a result of Bestvina--Bux--Margalit~\cite[Lemma 7.7]{BBM}, and is ingredient one in the proof of Proposition \ref{oneboundsymplemma}.

\begin{lemma}\label{lerayserrelemma}
Let $F$ be a non--cyclic free group and let
\begin{displaymath}
1 \rightarrow F \rightarrow G \rightarrow Q \rightarrow 1
\end{displaymath}
\noindent be a short exact sequence of groups.  Equip $H_1(F;\Z)$ with the $Q$--action induced by the conjugation action of $G$ on $F$.  Let $V \subseteq H_1(F;\Z)$ be a $Q$--invariant subgroup such that $H_1(F;\Z)/V$ is torsion free.  Suppose that:
\begin{enumerate}[label=(\arabic{enumi})]
\item $H_1(F;\Z)$ admits a nondegenerate $Q$--invariant symplectic form $\langle \cdot, \cdot \rangle$,
\item the restriction of $\langle \cdot, \cdot \rangle$ to $V$ is nondegenerate,
\item the quotient $H_1(F;\Z)/V$ is a trivial $Q$--module, and
\item $H_{\cd(Q)}(Q;\Z)$ is infinitely generated.
\end{enumerate}
\noindent Then $H_{\cd(Q) + 1}(G;\Z)$ is infinitely generated.  
\end{lemma}

\begin{proof}
We begin with the following claim.

\p{Claim} There is an isomorphism $H_{\cd(Q) + 1}(G;\Z) \cong H_{\cd(Q)}(Q;H_1(F;\Z))$. 

\p{Proof of claim}  The cohomological dimension of $F$ is 1 since $F$ is free.  Consider the Lyndon--Hochschild--Serre spectral sequence $\bE_{\cd(Q),1}^2$  \cite[Section VII.6]{Brownbook} corresponding to the short exact sequence
\begin{displaymath}
1 \rightarrow F \rightarrow G \rightarrow Q \rightarrow 1.
\end{displaymath}  
\bn This is the spectral sequence with $\bE_{p,q}^2 = H_p(Q;H_q(F;\Z))$ that converges to $H_{p+q}(G;\Z)$.  We have $\bE_{p,q}^r = 0$ for all $p +q > \cd(Q) + 1$ and $r \geq 2$, since either $p > \cd(Q)$ or $q > \cd(F) = 1$.   Thus there are no non--trivial differentials into or out of $\bE_{\cd(Q), 1}^2$ and therefore $\bE_{\cd(Q),1}^2 \cong \bE_{\cd(Q),1}^\infty$.  By the definition of the Lyndon--Hochschild--Serre spectral sequence, $\bE_{\cd(Q),1}^2 = H_{\cd(Q)}(Q;H_1(F;\Z))$, so the claim holds.
\smallskip

By the claim it suffices to show that $H_{\cd(Q)}(Q;H_1(F;\Z))$ is infinitely generated.  By hypothesis~(2), the subspace $V^{\perp} \subseteq H_1(F;\Z)$ is isomorphic to $H_1(F;\Z)/V$ as a $Q$--module.  But the form $\langle \cdot, \cdot \rangle$ is $Q$--invariant by hypothesis (1), so $V \oplus V^{\perp}$ is a decomposition of $H_1(F;\Z)$ into $Q$--modules.  Therefore the short exact sequence
\begin{displaymath}
0 \rightarrow V \rightarrow H_1(F;\Z) \rightarrow H_1(F;\Z)/V \rightarrow 0
\end{displaymath}
\noindent splits as a sequence of $Q$--modules, and thus $H_{\cd(Q)}(Q;H_1(F;\Z))$ splits as a direct sum
\begin{displaymath}
H_{\cd(Q)}(Q;V) \oplus H_{\cd(Q)}(Q;H_1(F;\Z)/V).
\end{displaymath}
\noindent  By hypothesis (3), the action of $Q$ on $H_1(F;\Z)/V$ is trivial.  Therefore there is an isomorphism
\begin{displaymath}
H_{\cd(Q)}(Q; H_1(F;\Z)/V) \cong H_{\cd(Q)}(Q;\Z) \otimes H_1(F;\Z)/V.
\end{displaymath}
\noindent  Since $H_{\cd(Q)}(Q;\Z)$ is infinitely generated by hypothesis, $H_{\cd(Q) + 1}(G;\Z)$ contains an infinitely generated free abelian subgroup and thus is infinitely generated as well.
\end{proof}

\p{Homology of covers} We now discuss the second ingredient of the proof of Proposition \ref{birmanhomolprop}.

\begin{lemma}\label{nondegencoverlemma}
Let $S = S_{g}^b$ with $b \in \{0,1\}$.  Let $\Gamma \trianglelefteq \pi_1(S)$ be an infinite index normal subgroup contained in $\pi_1^{(1)}(S)$.  Suppose that $\pi_1(S)/\Gamma$ is solvable and non--cyclic.  Let $S_\Gamma$ be the cover of $S$ corresponding to $\Gamma$.  Suppose that, if $\delta$ is a loop surrounding a boundary component of $S$, then $\delta \not \in \Gamma$.  Then the algebraic intersection form on $H_1(S_{\Gamma};\Z)$ is non--degenerate.
\end{lemma}

\begin{proof}
We use the fact that solvable, non--cyclic groups are one--ended (see, e.g., L{\"o}h \cite[Theorem 8.2.14]{Loh}).  Choose a basepoint $*_{\Gamma} \in S_{\Gamma}$ and a finite set of generators $\calF \subseteq \pi_1(S)$.  Let $G \subseteq S_{\Gamma}$ be an embedding of the Cayley graph of $\pi_1(S)/\Gamma$ containing $*_{\Gamma}$ such that each edge of $G$ is a lift of a loop $\gamma \in \calF$.  Since $S$ is compact, basic covering space theory tells us that $S_{\Gamma}$ is quasi--isometric to $G$.  Therefore $S_{\Gamma}$ has one end, so if $\delta \subseteq \pi_1(S)$ is a separating simple closed curve, at least one connected component of $S \cut \delta$ has compact closure in $S$.  But since no loop surrounding a boundary component is contained in $\Gamma$,  the surface $S_{\Gamma}$ has no compact boundary components.  Therefore $\delta$ is the boundary of a compact subsurface of $S_{\Gamma}$ given by the closure in $S_{\Gamma}$ of the finite--type connected component of $S_{\Gamma} \setminus \delta$.  In particular, this implies that any class in $H_1(S_{\Gamma};\Z)$ represented by a separating curve is the trivial class, so any nonzero class $v \in H_1(S_{\Gamma};\Z)$ is a multiple of a class represented by a nonseparating curve $a$.  Since $a$ is nonseparating, there is a curve $b \subseteq S_{\Gamma}$ with $\left|b \cap a\right| = 1$.  Therefore the class $v$ has nonzero algebraic intersection with the class $[b]$, so the algebraic intersection form on $H_1(S_{\Gamma};\Z)$ is nondegenerate, as desired.
\end{proof}

\bn We need one more fact about the homology of covers.

\begin{lemma}\label{inforbit}
Let $S = S_{g}^b$ with $b \in \{0,1\}$.  Let $\Gamma \trianglelefteq \pi_1(S)$ be an infinite index normal subgroup such that $\pi_1(S)/\Gamma$ is solvable and non--cyclic.  The subgroup of $H_1(\Gamma;\Z)$ spanned by the $\pi_1(S)$--orbit of any nonzero $v \in H_1(\Gamma;\Z)$ is infinitely generated. 
\end{lemma}

\begin{proof}
Let $S_{\Gamma} \rightarrow S$ be the regular cover corresponding to the subgroup $\Gamma \subseteq \pi_1(S)$.  Since $\Gamma$ has infinite index in $\pi_1(S_g)$, the cover $S_{\Gamma} \rightarrow S$ is infinite-sheeted.  By way of contradiction, suppose that there is a nonzero $v \in H_1(S_{\Gamma};\Z)$ such that the span of the orbit $\pi_1(S) v$ is finitely generated.  Since $\pi_1(S)v$ is finitely generated if and only if the span of $\pi_1(S)( m v)$ is finitely generated for any $m \in \Z$, we may assume that $v$ is primitive.

Let $\delta \subseteq S_\Gamma$ be a simple closed curve representing $v$.  Such a $\delta$ exists since $v$ is primitive.  By Lemma~\ref{nondegencoverlemma}, $H_1(S_{\Gamma};\Z)$ is nondegenerate, so there is another simple closed curve $\varepsilon \subseteq S_\Gamma$ such that $\left|\delta \cap \varepsilon\right| = 1$.  Let $T \subseteq S_\Gamma$ be an embedded copy of $S_1^1$ that contains $\delta$ and $\varepsilon$.  Since $T$ is compact and the action of $\pi_1(S_g)/\Gamma$ on $S_{\Gamma}$ is properly discontinuous, there are infinitely many disjoint $\pi_1(S_g)/\Gamma$--translates of $T$.  Hence there are infinitely many elements in the orbit $(\pi_1(S)/\Gamma)  \delta$ that are all supported on disjoint embedded copies $S_{1}^1 \subseteq S_\Gamma$.  Therefore there are infinitely many classes in the orbit $\pi_1(S)/\Gamma  [v]$ that are pairwise orthogonal in $H_1(S_\Gamma;\Z)$, so the lemma holds.  
\end{proof}

\bn We are almost ready to prove Proposition \ref{birmanhomolprop}.  Let $S_{g,n}$ denote an orientable surface of genus $g$ with $n$ punctures.  If $n = 1$ and $* \in S_{g,1}$ is a basepoint, we define $\Mod_{(k)}(S_{g-1,1})$ to be the subgroup of $\Mod_{(k)}(S_{g-1,1})$ given by the kernel of the natural map
\begin{displaymath}
\Mod(S_{g-1,1}) \rightarrow \Out(\pi_1(S_{g-1,1},*)/\pi_1^{(k)}(S_{g-1,1}^*)).
\end{displaymath}

\bn Akin \cite{Akin} proved that there is a short  exact sequence 
\begin{displaymath}
1 \rightarrow \pi_1^{(k-1)}(S_{g-1}) \rightarrow \Mod_{(k)}(S_{g-1,1}) \rightarrow \Mod_{(k)}(S_{g-1}) \rightarrow 1.
\end{displaymath}
\bn which we call the \textit{Akin--Birman exact sequence}.  We will apply Lemma \ref{lerayserrelemma} to Akin's sequence to prove the following.

\begin{lemma}\label{akinlemma}
Let $g \geq 3$, $k \geq 2$.  Suppose that $H_{2g-5}(\Mod_{(k)}(S_{g-1});\Z)$ is infinitely generated.  Then $H_{2g-4}(\Mod_{(k)}(S_{g-1,1};\Z))$ is infinitely generated.
\end{lemma}

\begin{proof}
We will show that the Akin--Birman exact sequence satisfies the hypotheses of Lemma~\ref{lerayserrelemma}.  Let $V \subseteq H_1(\pi_1^{(k-1)}(S_{g-1});\Z)$ be the image of the map
\begin{displaymath}
H_1(\pi_1^{(k)}(S_{g-1});\Z) \rightarrow H_1(\pi_1^{(k-1)}(S_{g-1});\Z).
\end{displaymath}
\bn Since $\pi_1^{(k)}(S_{g-1})$ is a characteristic subgroup, the subgroup $V$ is $\Mod_{(k)}(S_{g-1})$--invariant.  We will verify each hypothesis of Lemma \ref{lerayserrelemma} in turn.
\begin{enumerate}[label=(\arabic{enumi})]
\item \textit{$H_1(\pi_1^{(k-1)}(S_{g-1});\Z)$ admits a nondegenerate $\Mod_{(k)}(S_{g-1})$--invariant symplectic form $\langle \cdot, \cdot \rangle$.}  

Let $\Gamma = \pi_1^{(k-1)}(S_{g-1})$.  Let $S_{\Gamma}$ denote the cover of $S_{g-1}$ with $\pi_1(S_{\Gamma}) \cong \Gamma$.  There is a canonical isomorphism $H_1(S_\Gamma;\Z) \cong H_1(\Gamma;\Z)$ induced by the canonical map $\Gamma \rightarrow H_1(S_{\Gamma};\Z)$.  The symplectic form on $H_1(\Gamma;\Z)$ will be the pushforward of the algebraic intersection form $\langle \cdot, \cdot \rangle$ on $H_1(S_{\Gamma};\Z)$ under this isomorphism.  We will show that the $\Mod_{(k)}(S_{g-1})$--action on $H_1(\Gamma;\Z)$ given by the $\Mod_{(k)}(S_{g,1})$--conjugation on $\Gamma$ respects this symplectic form.  

Choose a basepoint $b \in S_{g-1}$ and a lift of the basepoint $b$ to $b_\Gamma \in S_{\Gamma}$.  The Birman exact sequence coupled with the Dehn--Nielsen--Baer theorem gives an identification 
\begin{displaymath}
\Mod(S_{g-1,1}) \cong \Aut(\pi_1(S_g)).
\end{displaymath}
\bn  The subgroup $\Gamma$ is characteristic, so there is a map
\begin{displaymath}
\Aut(\pi_1(S_{g-1})) \rightarrow \Aut(\Gamma).
\end{displaymath}
\bn By composing with the map $\Aut(\Gamma) \rightarrow \Out(\Gamma)$, there is a natural map
\begin{displaymath}
\Aut(\pi_1(S_{g-1})) \rightarrow \Out(\Gamma).
\end{displaymath}
\bn Since $\Gamma$ is characteristic, the choice of basepoints $b$ and $b_\Gamma$ coupled with the lifting criterion for covering spaces yields a map
\begin{displaymath}
\Mod(S_{g-1,1}) \rightarrow \Mod(S_\Gamma).
\end{displaymath} 
\bn Hence the algebraic intersection form on $H_1(\Gamma;\Z)$ is a $\Mod(S_{g-1,1})$--invariant symplectic form, and by restriction $\Mod_{(k)}(S_{g-1,1})$--invariant.  By the Akin--Birman exact sequence \cite{Akin}, the map
\begin{displaymath}
\Mod_{(k)}(S_{g-1,1}) \rightarrow \Mod_{(k)}(S_{g-1})
\end{displaymath}
\bn is surjective.  Hence for any $\varphi \in \Mod_{(k)}(S_{g-1})$, there is a lift $\overline{\varphi} \in \Mod_{(k)}(S_{g-1,1})$.  But if $\widehat{\varphi}$ is another such lift, the Akin--Birman exact sequence says that $\overline{\varphi}^{-1} \widehat{\varphi} \in \pi_1^{(k)}(S_{g-1}) = \Gamma$.  Therefore the composition $\overline{\varphi}^{-1} \widehat{\varphi}$ acts trivially on $H_1(S_{\Gamma};\Z) $ since $H_1(S_{\Gamma};\Z) \cong H_1(\Gamma;\Z)$, and so $\Mod_{(k)}(S_{g-1})$ acts in a well--defined way on $H_1(S_{\Gamma};\Z)$ and respects the nondegenerate form $\langle \cdot, \cdot \rangle$ on $H_1(S_{\Gamma};\Z)$.  Thus the $\Mod_{(k)}(S_{g-1})$--action on $H_1(\Gamma;\Z)$ respects the pushforward of the $H_1(S_\Gamma;\Z)$--intersection form.  The nondegeneracy of the form follows from Lemma \ref{nondegencoverlemma}, so hypothesis (1) holds.
\medskip
 
\item \textit{The restriction of $\langle \cdot, \cdot \rangle$ to $V$ is nondegenerate}

Suppose by way of contradiction that there is a non--zero $v \in V$ such that the linear map $\langle v, \cdot \rangle: H_1(\pi_1^{(k-1)}(S_{g-1});\Z) \rightarrow \Z$ is the zero map.  By Lemma~\ref{inforbit}, the $\pi_1(S_{g-1})$--orbit of $v$ in $H_1(\pi_1^{(k-1)}(S_{g-1});\Z)$ is infinitely generated.  Since $\langle \cdot, \cdot \rangle$ is $\pi_1(S_{g-1,1})$--invariant, any $\gamma v$ is degenerate for $\gamma \in \pi_1(S_{g-1,1})$.  Now, choose $w \in H_1(\pi_1^{(k-1)}(S_{g-1});\Z)$ such that $\langle v, w \rangle = 1$.  By Lemma \ref{inforbit}, the span of the $\pi_1(S_{g-1})$ orbit of $w$ is infinitely generated.  But $\langle \gamma w, \gamma v \rangle \neq 0$ for any $\gamma \in \pi_1(S_{g-1,1})$.  Since $\gamma v$ is degenerate in $V$ for any $\gamma \in \pi_1(S_{g-1,1})$, it must be the case that $\Span(\{\gamma w: \gamma \in \pi_1(S_{g-1})\})$ embeds into $H_1(\pi_1^{(k-1)}(S_{g-1};\Z))/V$.  But this latter space is finitely generated, which is a contradiction.  
\medskip

\item \textit{The quotient $H_1(\pi_1^{(k-1)};\Z)/V$ is a trivial $\Mod_{(k)}(S_{g-1})$--module.}

Since $\Mod_{(k)}(S_{g-1})$ acts trivially on $\pi_1(S_{g-1})/\pi_1^{(k)}(S_{g-1})$ by definition, $\Mod_{(k)}(S_{g-1})$ acts trivially on $H_1(\pi_1^{(k-1)}(S_{g-1});\Z)/V$ as well.

\medskip

\item \textit{$H_{\cd(\Mod_{(k)}(S_{g-1}))}(\Mod_{(k)}(S_{g-1});\Z)$ is infinitely generated.}  

By hypothesis, $H_{2g-5}(\Mod_{(k)}(S_{g-1});\Z)$ is infinitely generated.  Bestvina--Bux--Margalit \cite{BBM} showed that $\cd(\Mod_{(2)}(S_{g-1})) = 2(g-1) - 3 = 2g - 5$.  Hence 
\begin{displaymath}
\cd(\Mod_{(k)}(S_{g-1}) \leq \cd(\Mod_{(2)}(S_{g-1})) \leq 2g - 5.
\end{displaymath}
\bn Therefore if $H_{2g-5}(\Mod_{(k)}(S_{g-1};\Z)) \neq 0$, we have $\cd(\Mod_{(k)}) = 2g - 5$, so hypothesis~(4) of Lemma \ref{lerayserrelemma} holds.
\end{enumerate}
\bn Hence the hypotheses of Lemma~\ref{lerayserrelemma} hold for the Akin--Birman sequence, so $H_{2g-4}(\Mod_{(k)}(S_{g-1,1});\Z)$ is infinitely generated.
\end{proof}

\bn We are now almost ready to prove the main result of the section.  Before doing so, we will need one more auxiliary lemma which relates the slight variations of $\Mod_{(k)}$ that we have used throughout the paper.

\begin{lemma}\label{cleanuplemma}
Let $g \geq 3$ and $k \geq 3$.  Let $\Sigma_1 = (S_{g-1}^1, \{\{b_0\}\})$ and $\Sigma_2 = (S_{g-1}^2, \{\{b_0, b_1\}\})$.  Then the following hold:
\begin{enumerate}
\item $\Mod_{(k)}(S_{g-1,1}) \cong \Mod_{(k)}(\Sigma_1)$, and 
\item if $c \subseteq S_g$ is a nonseparating simple closed curve $\Stab_c(\Mod_{(k)}(S_g)) = \Mod_{(k)}(\Sigma_2)$.
\end{enumerate}
\end{lemma}
\begin{proof}
We prove each of these in turn.

\p{Part (1)} Church \cite[Definition 4.1]{Churchorbit} defines $\Mod_{(k)}(\Sigma_1)$ to be the kernel of the action of $\Mod(S_{g-1}^1)$ on the group $\pi_1(S_{g-1}^1, *)/\pi_1^{(k)}(S_{g-1}^1,*)$, where $*$ is a basepoint contained in $\partial S_{g-1}^1$.  Therefore there is a natural map
\begin{displaymath}
\Mod_{(k)}(\Sigma_1) \rightarrow \Mod_{(k)}(S_{g-1,1})
\end{displaymath}

\bn that is surjective and has kernel contained in $\langle T_{\delta} \rangle$, where $T_{\delta}$ is the twist around the boundary component of $S_{g-1}^1$.  But no power of $T_{\delta}$ is contained in $\Mod_{(k)}(\Sigma_1)$ for $k \geq 3$, so the map $\Mod_{(k)}(\Sigma_1) \rightarrow \Mod_{(k)}(S_{g-1,1})$ is an isomorphism.

\p{Part (2)} Let $f \in \Mod(S_{g-1}^2)$ denote the bounding pair map corresponding to the bounding pair consisting of the two boundary components of $S_{g-1}^2$.  A result of Birman--Lubotzky--McCarthy \cite[Lemma 2.1]{BLM} says that there is a short exact sequence

\begin{displaymath}
1 \rightarrow \langle f \rangle \rightarrow \Mod(S_{g-1}^2) \rightarrow \Stab_c(\Mod(S_g)) \rightarrow 1.
\end{displaymath} 

\bn The latter map $\Mod(S_{g-1}^2) \rightarrow \Stab_c(\Mod(S_g))$ is induced by the embedding $S_{g-1}^2 \rightarrow S_g$ such that the complement of the image of the embedding contains $c$.  In particular, the work of Church \cite[Theorem 4.6]{Churchorbit} says that this short exact sequence restricts to a short exact sequence

\begin{displaymath}
1 \rightarrow F \rightarrow \Mod_{(k)}(\Sigma_2) \rightarrow \Stab_c(\Mod_{(k)}(S_g)) \rightarrow 1
\end{displaymath}

\bn for $F \subseteq \langle f \rangle$.  But then no power of $f$ is contained in $\Mod_{(k)}(\Sigma_2)$ for $k \geq 3$, so $F = \{1\}$ and hence $\Mod_{(k)}(\Sigma_2) \cong \Stab_c(\Mod_{(k)}(S_g))$.
\end{proof}

\bn Recall that our goal is to show that $H_{2g-5}(\Mod_{(k)}(S_{g-1};\Z))$ is infinitely generated, then the group $H_{2g-3}(\Stab_c(\Mod_{(k)}(S_{g}));\Z))$ is infinitely generated as well, for $c$ a non--separating simple closed curve.  

\medskip

\begin{proof}[Proof of Proposition \ref{birmanhomolprop}]
Since $k \geq 3$, Lemma \ref{cleanuplemma} says that we have isomorphisms $\Mod_{(k)}(\Sigma_1) \cong \Mod_{(k)}(S_{g-1,1})$ and $\Mod_{(k)}(\Sigma_2) \cong \Stab_c(\Mod_{(k)}(S_g))$.  Hence, it suffices to show that if the group $H_{2g-4}(\Mod_{(k)}(\Sigma_1);\Z)$ is infinitely generated, then $H_{2g-3}(\Mod_{(k)}(\Sigma_2);\Z)$ infinitely generated.  We will apply Lemma \ref{lerayserrelemma} to the short exact sequence from Proposition \ref{oneboundsymplemma}.  We take $V \subseteq H_1(F_k;\Z)$ to be the image of $F_{k+1}$ in $H_1(F_k;\Z)$.  We verify each hypothesis for Lemma \ref{lerayserrelemma} in turn.
 
\begin{enumerate}[label=(\arabic{enumi})]

\item \textit{$H_1(F_k;\Z)$ admits a nondegenerate $\Mod_{(k)}(\Sigma_1)$--invariant symplectic form} 

This form is the algebraic intersection form on the cover of $S_{F_k} \rightarrow S_{g-1}^1$ corresponding to $F_k$.  The group $\Mod_{(k)}(\Sigma_1)$ acts on $H_1(S_{F_k};\Z)$ as a consequence of Lemma \ref{conjactionlemma}.  This action respects the algebraic intersection form on $H_1(S_{F_k};\Z)$, since the action of $\varphi \in \Mod_{(k)}(\Sigma_1)$ on $H_1(S_{F_k};\Z)$ is the pushforward of a homeomorphism $S_{F_k} \rightarrow S_{F_k}$.  The group $\pi_1(S_{g-1}^1)/F_k$ is solvable and non--cyclic by property (3) of Lemma \ref{oneboundsymplemma}.  Furthermore, if $\delta$ is the loop around the boundary component of $S_{g-1}^1$, then $\delta \not \in F_k$ for $k \geq 3$, since the push along $\delta$ induces a Dehn twist along a separating curve.  Therefore the  algebraic intersection form on $H_1(F_k;\Z)$ is non--degenerate by Lemma \ref{nondegencoverlemma}.

\item \textit{The restriction of $\langle \cdot, \cdot \rangle$ to $V$ is nondegenerate} 

The form $\langle \cdot, \cdot \rangle$ is non--degenerate when restricted to $V = \im(F_{k+1} \rightarrow H_1(F_k;\Z))$.  In particular, Proposition \ref{oneboundsymplemma} says that $H_1(F_k;\Z)/V$ is nontrivial, finitely generated, and free abelian.  Then Lemma \ref{inforbit} of says that the $\pi_1(S_{g-1}^1)$ orbit of any nontrivial element in $H_1(F_k;\Z)$ spans an infinitely generated subgroup.  The argument is complete by the same argument as step (2) of Lemma \ref{akinlemma}.

\item \textit{The quotient $H_1(F_k;\Z)/V$ is a trivial $\Mod_{(k)}(\Sigma_2)$--module}

This is property (2) of $F_k$ from Proposition \ref{oneboundsymplemma}.

\item \textit{The group $H_{\cd(\Mod_{(k)}(\Sigma_1))}(\Mod_{(k)}(\Sigma_1);\Z)$ is infinitely generated.} 

By monotonicity and the fact that $\cd(\Mod_{(2)}(S_{g-1})) = 2g - 3$ \cite[Theorem B]{BBM}, we have $\cd(\Mod_{(k)}(S_{g-1}) \leq 2g - 3$.  By applying the sub--additivity of cohomological dimension on short exact sequences \cite[Section VIII.2]{Brownbook} to the Akin--Birman exact sequence,  we have $\cd(\Mod_{(k)}(S_{g-1,1})) \leq 2g - 4$.  Since $H_{2g-4}(\Mod_{(k)}(S_{g-1,1});\Z)$ is non--zero by Lemma \ref{akinlemma}, we must have
\begin{displaymath}
\cd(\Mod_{(k)}(S_{g-1,1})) = 2g - 4.
\end{displaymath}
\bn Then Lemma \ref{akinlemma} says that $H_{2g-4}(\Mod_{(k)}(S_{g-1,1}))$ is infinitely generated, so the hypothesis is satisfied.
\end{enumerate}
\bn Therefore $H_{2g-3}(\Mod_{(k)}(\Sigma_2);\Z)$ is infinite dimensional by Lemma \ref{lerayserrelemma}.  Since $\Mod_{(k)}(\Sigma_2)$ is isomorphic to $\Stab_c(\Mod_{(k)}(S_g))$ by Lemma \ref{cleanuplemma}, the proof is complete.  
\end{proof}

\section{The proofs of Theorems~\ref{mainthm} and~\ref{infgenthm}}\label{mainsection}

In this section, we complete the proofs of Theorems \ref{mainthm} and \ref{infgenthm}.  We begin by briefly discussing the complex of cycles, $\calB(S_g)$, a cell complex defined by Bestvina--Bux--Margalit \cite[Section 2]{BBM}.  The following properties of $\calB(S_g)$ are the relevant properties for the purposes of this paper:

\begin{enumerate}[label=\textit{(\alph{enumi})}]
\item $\calB(S_g)$ is contractible \cite[Theorem E]{BBM},
\item there is a vertex $v \in \calB(S_g)$ represented by a nonseparating simple closed curve $c \subseteq S_g$, and
\item $\calB(S_g)$ admits a rotation--free $\cI_g$--action \cite[Corollary 1.8]{Ivanovcomplex}.
\end{enumerate}
As discussed in the introduction, we will adopt Bestvina--Bux--Margalit's strategy for their proof that $H_{3g-5}(\cI_g;\Z)$ is infinitely generated \cite[Section 8]{BBM}.  Specifically, we will consider the action of $\cI_g$ on the complex $\calB(S_g)$.  We will apply the equivariant homology spectral sequence \cite[Section VII.7]{Brownbook}.  We will use Proposition~\ref{birmanhomolprop} and the second property of $\calB(S_g)$ above to show that the entry $\bE_{0,2g-3}^{\infty}$ is infinitely generated.  We begin with the following result.  
\begin{lemma}\label{injstablemma}
Let $S = S_g$ with $g \geq 3$, and let $k \geq 3$.  Let $c \subseteq S_g$ be a nonseparating simple closed curve.  The pushforward map
\begin{displaymath}
H_{2g-3}(\Stab_{c}\Mod_{(k)}(S_g);\Z) \rightarrow H_{2g-3}(\Mod_{(k)}(S_g);\Z)
\end{displaymath}
\noindent is an injection.
\end{lemma}

\begin{proof}
Let $\bE_{*,*}^*$ be the equivariant homology spectral sequence given by the action of $\Mod_{(k)}(S_g)$ on $\calB(S_g)$ as discussed in the introduction.  For each $0 \leq k \leq 2g  - 3$, let $\Sigma_k$ be a set of representatives in $\calB(S_g)$ of the $k$--cells in the quotient $\calB(S_g)/\cI_g$.  By property \textit{(c)} of $\calB(S_g)$, $\Mod_{(k)}(S_g)$ acts on $\calB(S_g)$ without rotations.  Hence there is a decomposition
\begin{displaymath}
\bE_{p,q}^1 \cong \bigoplus_{\sigma \in \Sigma_p} H_q(\Stab_\sigma(\Mod_{(k)}(S_g)).
\end{displaymath}
\bn By property \textit{(b)} of $\calB(S_g)$, there is a vertex $v$ of $\calB(S_g)$ represented by a simple closed curve $c$, there is an injection
\begin{displaymath}
H_{2g-3}(\Stab_c (\Mod_{(k)}(S_g));\Z) \hookrightarrow \bE_{0,2g-3}^1.
\end{displaymath}
\noindent Bestvina--Bux--Margalit prove that for any $j$--cell $\sigma \subseteq \calB(S)$ \cite[Proposition 6.2]{BBM}, we have
\begin{displaymath}
\cd(\Stab_\sigma\Mod_{(2)}(S_g)) \leq 2g - 3 - j.
\end{displaymath}
\noindent  Since $H \subseteq G$ implies $\cd(H) \subseteq \cd(G)$ \cite[Section VIII.2]{Brownbook}, we have
\begin{displaymath}
\cd(\Stab_\sigma\Mod_{(k)}(S_g)) \leq 2g - 3 - j.
\end{displaymath}
\bn for any $j$--cell $\sigma \subseteq \calB(S)$.  Hence the non--zero entries of the first page of the spectral sequence $\bE_{p,q}^r$ are as in Figure \ref{specfig}.

\begin{figure}[h]
\begin{center}
\begin{tikzpicture}
\matrix (m) [matrix of math nodes,
             nodes in empty cells,
             nodes={minimum width=10ex,
                    minimum height=5ex,
                    outer sep=-5pt},
             text centered,anchor=center]{
    2g - 3\strut   &  \displaystyle{\bigoplus_{\sigma \in \Sigma_0} H_{2g-3}(\Stab_{\sigma}(\Mod_{(k)}(S_g)))}  & 0 & \cdots& \cdots  & 0 \\
    \vdots   & \vdots   &\vdots & \ddots       & \ddots & \vdots  \\      
    1      &  \displaystyle{\bigoplus_{\sigma \in \Sigma_0} H_{1}(\Stab_{\sigma}(\Mod_{(k)}(S_g)))}   &  \cdots & \cdots  &  0 & \vdots \\             
   0      &  \displaystyle{\bigoplus_{\sigma \in \Sigma_0} H_{0}(\Stab_{\sigma}(\Mod_{(k)}(S_g))) }  & \cdots & \cdots & \cdots &  0 \\
  \quad\strut &  0   &  1 & \cdots & 2g-3& {2g-2}\\};
\draw[thick] (m-1-1.north east) -- (m-5-1.east) ;
\draw[thick] (m-5-1.north) -- (m-5-6.north east) ;
\end{tikzpicture}
\end{center}
\caption{Page 1 of the sequence $\bE_{p,q}^r$}\label{specfig}
\end{figure}

\bn On the diagonal $p +q = 2g - 2$, every entry $\bE_{p,q}^1$ is $0$.  Therefore for every $r \geq 1$, the map $d_{r+1,2g-3 - r}^r$ mapping into $\bE_{0,2g-3}^{r}$ is the zero map.  Hence there is an injection
\begin{displaymath}
H_{2g-3}(\Stab_c \Mod_{(k)}(S_g);\Z) \hookrightarrow \bE_{0,2g-3}^\infty.
\end{displaymath}
\noindent  Property \textit{(a)} of $\calB(S_g)$ says that $\calB(S_g)$ is contractible.  Hence $\bE_{p,q}^r$ converges to $H_{p+q}(\Mod_{(k)}(S_g);\Z)$ \cite[Section V]{Brownbook}, so $\bE_{0,2g-3}^\infty$ injects into $H_{2g-3}(\Mod_{(k)}(S_g);\Z)$.  Hence the composition
\begin{displaymath}
H_{2g-3}(\Stab_c \Mod_{(k)}(S_g);\Z) \rightarrow \bE_{0,2g-3}^1 \cong \bE_{0,2g-3}^\infty \rightarrow H_{2g-3}(\Mod_{(k)}(S_g);\Z)
\end{displaymath}
\noindent is an injection, so the proof is complete.
\end{proof}
\noindent  We are now ready to prove Theorem~\ref{infgenthm}.

\medskip  

\begin{proof}[Proof of Theorem~\ref{infgenthm}]
Fix a $k \geq 3$.  We will prove the theorem by induction on $g$.

\p{Base case: $g = 2$}  As mentioned in the introduction, Mess showed that $\Mod_{(1)}(S_2)$ is an infinitely generated free group \cite{Messfree}.  Since the $k$th term of the lower central series $\Mod_{(1)}^{(k)}(S_2)$ is contained in $\Mod_{(k)}(S_2)$ for every $k \geq 2$, $\Mod_{(k)}(S_2)$ is an infinitely generated free group as well.  Hence $H_1(\Mod_{(k)}(S_2);\Z)$ is infinitely generated.

\p{Inductive step}  Suppose that the theorem holds for all $2 \leq g' < g$.  Let $c$ be a nonseparating simple closed curve on $S_g$.  By Proposition~\ref{birmanhomolprop} and the inductive hypothesis, $H_{2g-3}(\Stab_c(\Mod_{(k)}(S_g));\Z)$ is infinitely generated.  Then by Lemma~\ref{injstablemma}, $H_{2g-3}(\Mod_{(k)}(S_g);\Z))$ is infinitely generated as well, so the proof is complete.
\end{proof}
\noindent We now prove the main result of the paper.

\medskip

\begin{proof}[Proof of Theorem~\ref{mainthm}]
By Theorem~\ref{infgenthm} and the universal coefficient theorem, \begin{displaymath}
\cd(\Mod_{(k)}(S_g)) \geq 2g - 3.
\end{displaymath}
\bn  Bestvina--Bux--Margalit  showed that $\cd(\Mod_{(2)}(S_g)) = 2g - 3$ \cite[Theorem B]{BBM}.  By the monotonicity of cohomological dimension, we have 
\begin{displaymath}
\cd(\Mod_{(2)}(S_g)) \geq \cd(\Mod_{(k)}(S_g))
\end{displaymath}
\bn for $k \geq 2$.  Hence the equality $\cd(\Mod_{(k)}(S_g)) = 2g - 3$ holds.
\end{proof}

\bibliographystyle{alpha}
\bibliography{C:/Users/dsmin/OneDrive/Documents/Documents/Math/Bibliography/mainbib}

\end{document}